\documentclass[10pt, reqno]{amsart}
 \usepackage{amsmath}
 \usepackage{amssymb}
\usepackage{bm}

\DeclareMathAccent{\mathring}{\mathalpha}{operators}{"17}

\newcommand{\mysection}[1]{\section{#1}
      \setcounter{equation}{0}}

\newtheorem{theorem}{Theorem}[section]
\newtheorem{lemma}[theorem]{Lemma}
\newtheorem{proposition}[theorem]{Proposition}
\newtheorem{corollary}[theorem]{Corollary} 

\theoremstyle{definition}
\newtheorem{assumption}[theorem]{Assumption}
\newtheorem{definition}[theorem]{Definition}
\theoremstyle{remark}
\newtheorem{remark}[theorem]{Remark}

 \makeatletter 
 \def\dashint{%  
 \operatorname%
 {\,\,\text{\bf--}\kern-.98em\DOTSI\intop\ilimits@\!\!}}
\def\ninf{\qopname\relax\@empty{inf\phantom{p}\!\!\!}}
 \makeatother

\newcommand\bbeta{\text{\raise-.2ex\hbox{$\bm{\beta}$}}}

\newcommand\esssup{\operatornamewithlimits{esssup}}

\newcommand\bR{\mathbb{R}}

\newcommand\cF{\mathcal{F}}

\newcommand\cN{\mathcal{N}}

\begin{document}

\title[Strong solutions via approximations]
{Existence of strong solutions for It\^o's 
stochastic  equations via approximations. Revisited}

\author{I. Gy\"ongy}
\email{I.Gyongy@ed.ac.edu}
 \address{School of Mathematics and Maxwell Institute, University of Edinburgh, Scotland, UK}
 
\author{N.V. Krylov}
 
\email{nkrylov@umn.edu}
\address{127 Vincent Hall, University of Minnesota,
 Minneapolis, MN, 55455}



\begin{abstract}
Given  strong uniqueness for an
It\^o's 
stochastic equation, we prove that its solution can be
constructed on ``any'' probability space by using, for example,
 Euler's polygonal approximations. Stochastic equations in
$\bR^{d}$ and in domains in $\bR^{d}$
are considered. This is almost a copy
of an old article in which we correct
 errors
  in the  original proof of Lemma 4.1 
found by  Martin Dieckmann in 2013. 
 We present also a new result on the 
convergence of ``tamed Euler approximations" for SDEs with 
locally unbounded drifts, which we achieve by proving an estimate for 
appropriate exponential moments.

\end{abstract}

\maketitle

\mysection{Introduction}
We start with two examples illustrating the results we present in the paper.
Consider the stochastic differential equation
\begin{equation}
                                         \label{6.4.1}
dx(t)=\Big[\tan\big(-\tfrac{\pi}{2} x(t)\big
)+1\Big]
\,dt+{|1-|x(t)\|}^{\alpha}(x_+(t))^{\frac {1}{2}}\,dw(t),\quad  
x(0)=0   
\end{equation}
with a given $\alpha >0$, where $w$ is a Wiener process. 
Note that the 
drift coefficient $\tan (-\frac{\pi}{2} x)$ 
is not continuous at $x=2k+1$, for integers $k$,  and it does not 
satisfy the linear growth condition. Note moreover that the diffusion
coefficient
$|1-|x\|^{\alpha}(x_+)^{\frac{1}{2}}$ does not satisfy the linear
growth condition for $\alpha >\frac{1}{2} $ and it is not H\"older
continuous with exponent $1/2$ if $\alpha <1/2$. Consider also
the equation
\begin{equation} 
                                                \label{6.4.2} 
dx(t)=\Big[\tan\big(-\tfrac{\pi}{2} x(t)\big) +\text{\rm sign}\,
x(t) \Big] \,dt+{|1-|x(t)\|}^{\alpha}\,dw(t),\quad 
x(0)=0, 
\end{equation}
 and note that here the drift is discontinuous also at $0$.

The coefficients in the above equations are rather irregular, one can
define, however,  Euler's ``polygonal'' approximations: 
\begin{equation}
                                           \label{6.4.3}
dx_n(t)=b(x_n(\kappa_n(t)))\,dt+\sigma(x_n(\kappa_n(t)))\,dw(t),\quad  
x_n(0)=0 
\end{equation}
 for every integer $n>0$, where $\kappa_n(t):= \lfloor nt\rfloor/n$, with the
corresponding drift  and diffusion coefficient, setting for example
$b(x)=0$ when $x$ is an odd integer.  One expects that in each of
these examples $x_n$ converges in probability to a process which
solves the corresponding equation \eqref{6.4.1} and \eqref{6.4.2}, respectively.

In fact, the drift in equation \eqref{6.4.1} is Lipschitz continuous and the
diffusion coefficient is H\"older continuous with exponent
$\frac{1}{2}$ at $x=0$.  Therefore by a  well-known result of Yamada
and Watanabe \cite{15.}  one knows the existence (at least of a local)
strong solution to equation \eqref{6.4.1}. In the case of equation \eqref{6.4.2} one
can say the same due to a result of Veretennikov \cite{14.}, stating the existence of a
unique strong solution to the stochastic differential equation
\begin{equation}
                                           \label{6.4.4}
 dx(t)=b(t,x(t))\,dt+\sigma (t,x(t))\,dw(t),\qquad x_0\in 
\bR^d                     
\end{equation}
 in $\bR^d$, with a given $d_1$-dimensional Wiener process $w$,
if $b$, $\sigma$  are bounded measurable functions on $\Bbb
R_+\times\Bbb R^d$  with values in $\Bbb R^d$ and in $\Bbb R^{d\times
d_1}$, respectively; $\sigma {\sigma}^{ \ast}$  is uniformly elliptic;
$\sigma$ is H\"older continuous in $x\in \Bbb R$ with exponent
$\frac{1}{2}$ when $d=1$,  and it is Lipschitz in $x\in \Bbb R^d$ in
the multidimensional case.

The method of  establishing  these existence and uniqueness
theorems  is rather different from those used in the theory of
ordinary differential  equations. It is based on a famous result from
Yamada and Watanabe  \cite{15.} which reads  as follows. If any two
solutions to equation \eqref{6.4.4}
 on the same probability space with the
same  Wiener process almost surely coincide, and if there is a
solution to the equation   on some probability space with some 
suitable Wiener process, then there exists a strong solution to the
equation  with the given Wiener process. Shortly speaking, the
existence of a solution and the pathwise  uniqueness imply the
existence of the unique strong solution. (See also   Zvonkin and
Krylov \cite{16.}  and the references therein on this topic.) We emphasize
that by this approach one gets only pure existence result, without
presenting any construction of the solution. 

The existence of a solution to equation \eqref{6.4.4} with  bounded
measurable  coefficients is known under the additional condition that
$\sigma (t,x) $, $b(t,x)$ are continuous in $x$ (Skorokhod \cite{13.},
Stroock and  Varadhan \cite{12.}), or 
 $\sigma {\sigma}^{ \ast}$ is uniformly
elliptic (Krylov \cite{5.}, \cite{8.}), regarding recent progress
in the case of singular $b$ see \cite{Kr_20}
and the references therein. Hence  
  Veretennikov, Yamada and
Watanabe establish the existence of a strong solution 
(in  \cite{14.} and
in \cite{15.}, respectively) by proving the pathwise uniqueness. Their
proofs raise  the following questions. Is it possible to construct
the strong solutions in some classical way under the conditions of
their theorems? Define for example  Euler's approximations 
\eqref{6.4.3} to
equation \eqref{6.4.2}. Do these approximations
 converge to a stochastic
process in probability and can one construct a strong solution in 
this way?  Let us approximate  the coefficients in the equation 
\eqref{6.4.4}
by smooth ones. Do the strong solutions of  the corresponding
equations converge in probability to the strong solution of  equation
\eqref{6.4.4} under the  assumptions of the cited existence 
theorem? More
generally, does the  strong solution depend continuously, in the
topology of convergence  in probability, on the initial condition and
on the drift and diffusion coefficients? 

Our aim is to show that the answers to these questions are in the
affirmative. We prove, roughly speaking, that Euler's polygonal
approximations converge  uniformly in $t$ in bounded intervals, in
probability, to a process, which we show to be the strong solution,
if the pathwise uniqueness for the equation holds, provided the 
drift and  diffusion coefficients have some additional property
permitting the passage to the   limit. Such additional property is
their continuity in the space variable,  or the strong ellipticity of
the diffusion coefficient. (See Theorems 2.4 and 2.8 below.) In
particular, applying Corollaries 2.7 and 2.9 to equations 
\eqref{6.4.1} and
\eqref{6.4.2},  respectively with $D:=(-1,1)$, 
$D_k:=(-1+2^{-k},1-2^{-k})$ and with $V(t,x):=\frac{2-x^2}{1-x^2}$, 
we get that Euler's approximations $x_n(t)$, defined by 
\eqref{6.4.3}
converge uniformly in $t$ in bounded intervals in probability to some
stochastic processes, which are the  strong solutions of equations
\eqref{6.4.1} and \eqref{6.4.2} respectively.

Let us finally consider the following example of an 
SDE with singular drift 
\begin{equation}                                                                          \label{7.21.1}
dx(t)=|x(t)|^{-1/5}dt+(2+\sin(x(t)))\,dw (t) , \quad x(0)=0.
\end{equation}
By results on SDEs with locally unbounded drifts (see, e.g., 
\cite{10.}, \cite{2b}, \cite{15b},
\cite{Zh_20}, \cite{RZ_21_2}) one can see that this equation has a unique 
strong solution.  Note that the Euler approximations 
\eqref{6.4.3} are not meaningful, but we can define the ``tamed" Euler approximations
$$
dx_n(t)=b_n(x_n(\kappa_n(t))\,dt+(2+\sin((x_n(\kappa_n(t)))\,dt, \quad 
x_n(0)=0
$$
for example with $b_n(x)=|x|^{-1/5}\wedge\lambda_n$ for a sequence of positive constants 
$\lambda_n$ converging to infinity. Applying our result, Theorem \ref{theorem 6.21.5}
below, we get that if $\lambda_n$ converges to infinity sufficiently slowly, then 
the tamed Euler approximations converge to the solution $x(t)$ of equation \eqref{7.21.1},  
in probability, uniformly in $t$ in bounded intervals.

The possibility to show convergence of different approximations to
solutions of stochastic equations is based on the following simple
observation.

\begin{lemma}                                                             \label{lemma 6.4.1}

 Let $Z_n$ be a sequence of random elements in a
Polish space ($\Bbb E$, $\rho$)
 equipped with the   Borel $\sigma$--algebra. Then $Z_n$ converges in probability to an  
$\Bbb E$-valued random element if and only if
  for every
subsequences $Z_l$ and
$Z_m$ there exists a subsequence $v_k:=(Z_{l(k)},Z_{m(k)})$
converging weakly to a random  element $v$ supported on the diagonal
$
\{(x,y)\in \Bbb E\times \Bbb E\,:\,\, x=y\}.
$  
\end{lemma}

The necessity of the condition is obvious.
To prove the sufficiency it is enough to note that for the continuous
function
$f(x,y)=\rho(x,y)$ the random variables $f(v_{k})$ converge to
$f(v)=0$ weakly and hence, $f(v_{k})\to 0$  {\em in probability}. This implies that
$\{Z_n\}$ is a Cauchy sequence in the  space of random
$\Bbb{E}$-valued elements with the metric corresponding to
convergence in probability. Since this space is complete, our
assertion holds indeed.

In our applications of the lemma Skorokhod's embedding method and the
assumption about pathwise uniqueness will allow us to check that the
limiting random element $v$ takes values in 
$\{(x,y)\in \Bbb E\times \Bbb E\,:\,\, x=y\}$.
 
We note that our approach is very close in its spirit to the
celebrated result of Yamada and Watanabe on the existence of strong
solutions via pathwise uniqueness. We assume somewhat more and in
return we can get more. From our approach it is  clear that the
strong solution depends continuously on the initial condition and on 
the drift and diffusion coefficient. In particular, it can be seen in
the same way as Theorems 2.4 and 2.8 are proved that, the strong
solution can be constructed by approximating the coefficients by
smooth ones. One can construct the strong solution by Euler's
approximations and approximating simultaneously the coefficients and
the  initial condition. We remark, that clearly we immediately get
the  convergence of Euler's approximations (or of the other
approximations  we mentioned) in probability in every metric space
$V$, were  these approximations are tight. (See Gy\"ongy, Nualart and
Sanz-Sol\'e \cite{3.},  were the convergence in probability of Wong-Zakai
type approximations are proved  in the modulus spaces introduced
there.) 

Finally we remark that the convergence of Euler's approximations
under various   conditions is proved by many authors. It is shown in
Krylov \cite{7.} that under the monotonicity  condition Euler's  polygonal
line method can be adjusted to prove (strong) solvability. Earlier
this was known from Maruyama \cite{9.} when the drift and diffusion
coefficients are Lipschitz continuous. The method of \cite{7.} was
afterward used in Alyushina \cite{1.} in a short proof of existence of
strong solutions under the monotonicity and the linear growth
conditions. Later a short and  simple proof of (strong) solvability
is presented in Krylov \cite{6.} under the monotonicity  condition and
under a growth condition which is weaker than the usual linear growth
condition. Moreover, the continuous dependence of the strong solution
on the coefficients is also  obtained. 

It is also worth mentioning that the fact that the pathwise uniqueness implies the
possibility of effective constructing the solutions has already been noticed in
Zvonkin and Krylov \cite{16.} (see, for instance, Lemma 3.2 there). Later Kaneko
and Nakao \cite{KN_88} exploited this fact without noticing \cite{16.}. In \cite{KN_88} the authors
consider equation (1.3) in $\bR^{d}$ and they assume that it admits a unique strong
solution $x(t)$. They show that $x(t)$ can be constructed by approximating the
coefficients and also by Euler's polygonal approximation. In what concerns
Euler's approximations they only consider equations in the whole space with
continuous coefficients satisfying the linear growth condition. We consider
equations also in domains of $\bR^{d}$ and with discontinuous coefficients as well.
We construct the strong solution without assuming its existence. Our basic
idea of proving convergence in probability is an extension of the idea 
of another result of Yamada and Watanabe saying that pathwise uniqueness implies
uniqueness in law. Essentially the same idea is used in \cite{5.}. Due to our above
lemma this idea becomes more apparent and its range of applicability becomes
evident.

 We remark that since the original version of the present paper was published 
there has been a growing 
interest in studying  the convergence of Euler's approximations for SDEs with irregular coefficients.  
For recent results on the rate of convergence we refer to \cite{DG2020}, 
\cite{LSz}, \cite{NT2016}, 
and the references therein.  

The paper is organized as follows. In the next section we formulate
our results  Theorems \ref{Theorem 2.4}, \ref{Theorem 2.8}, 
\ref{theorem 6.21.5}
and their corollaries. By Lemma 1.1
the proof of Theorem \ref{Theorem 2.4} is simple,  we present it in Section 3. To
prove Theorem \ref{Theorem 2.8} we need an estimate of the distribution for Euler's
approximations. Since such estimates play an important role not only
in the subject of  the paper, we present our estimate (Theorem 4.2
below) separately in Section 4.  We prove our main results, Theorem
2.8 and \ref{theorem 6.21.5}, in the last two sections. 

\mysection{Formulation of the results}                                  \label{Section 2}

  On a given stochastic basis $(\Omega,\cF, P,  (\cF_t)_{t\geq
0})$ we consider the stochastic differential equation
\begin{equation}
                                                   \label{6.4.2.1}
dx(t)=b(t,x(t))\,dt+\sigma (t, x(t))\,dw(t),\quad   x(0)=\xi  
\end{equation}
 in a domain $D$ of $\Bbb R^d$, where $(w(t),\cF_t)$ is a
$d_{1}$-dimensional Wiener  process, $\xi$ is an $\cF_0$-measurable random vector with values in $D$, 
$b$ and $\sigma$ are Borel functions on $\Bbb R_+\times D$ taking
values  in $\Bbb R^d$ and in $\Bbb R^{d\times d_{1}}$, respectively. 
For equation \eqref{6.4.2.1} to have sense we need the coefficients to be
defined for any $x\in\Bbb R^d$. Actually under our future assumptions
solutions of (2.1) will never leave $D$ so that values of $\sigma$
and $b$ outside $D$ are irrelevant and just for convenience we define
$\sigma(t,x)=0,b(t,x)=0$ for $x\not\in D, t\geq0$. Let 
$$ 0=t_0^n<t_1^n<t_2^n<...<t_i^n<t_{i+1}^n<...
$$ be a sequence of partitions of $\Bbb R_+$  
 such that $\lim_{i\to\infty}t_i^n=\infty$ for every $n\geq1$, and  for every $T>0$
$$ d_n(T):=\sup_{i:\,t_{i+1}\leq T}|t_{i+1}^n-t_i^n| \to 0
$$ as $n\to \infty$. We define Euler's ``polygonal'' approximations
as the process $(x_n(t))$ satisfying
\begin{equation}
                                                                                                \label{6.4.2.2}
dx_n(t)=b(t,x_n(\kappa_n(t))\,dt+\sigma (t,
x_n(\kappa_n(t))\,dw(t),\quad   x_n(0)=\xi  
\end{equation}
where $\kappa_n(t):=t_i^n$ for $t\in [t_i^n,t^n_{i+1})$. 
 
In the whole article $M(t)>0$ and $M_{1}(t)>0,M_{2}(t)>0,...$ are
fixed
  locally integrable functions on $[0,\infty)$. We will use the
following assumptions:
 
 ($i$) there exists an increasing sequence of bounded
 domains $\{ D_k\}_{k=1}^\infty$ such that $\cup_{k=1}^\infty D_k=D$,
and for every $k$, $t\in [0,k]$
$$
\sup_{x\in D_k} |b(t,x)|\leq M_k(t),\qquad \sup_{x\in D_k} |\sigma
(t,x)|^{2}
\leq M_{k}(t);
$$ 
 ($ii$) there exists a non-negative function $V\in
C^{1,2}(\Bbb R_+\times D)$  such that 
 $$ 
LV(t,x)\leq M(t )V(t,x),\,\,\,\forall t\geq 0,\,x\in D,
$$   
$$
 V_{k}(T):=\inf_{x\in \partial D_k,t\leq T}|V(t,x)|\to\infty
$$   as  $k\to\infty$ for every finite $T$, where $\partial D_k$
denotes the boundary of $D_k$ and $L$ is the differential operator
$$ L:=\frac{\partial}{\partial t}+\sum_i
b_i(t,x)D_{i} +\frac {1}{2}
\sum_{i,j} (\sigma\sigma^{ *})_{ij}(t,x) D_{ij}
\quad D_{i}=\frac{\partial}{\partial x^i},\quad D_{ij}=D_{i}D_{j}, 
$$
 where $\sigma^{ *}$ denotes the transpose of the matrix $\sigma$; 
 
($iii$) $P(\xi \in D)=1$.

Note that by $(i)$ and by our definition of $\sigma$ and $b$ outside
$D$,
 Euler's approximations $x_{n}(t)$ are well defined for all $t\geq0$.

\begin{definition}
                                       \label{definition 6.4.2.1}

 By solution of equation \eqref{6.4.2.1} we mean an
${\cF}_{t}$--adapted process
$x(t)$ which does not ever leave $D$ and satisfies (2.1).
\end{definition}

An explanation of the definition can be found in the following
statement.

\begin{lemma}
                                           \label{lemma 6.4.2.2}
  Let $x(t)$ be an ${\cF}_{t}$--adapted
process defined for all $t\geq0$. Assume that
$x(t)$ satisfies (2.1) for $t<\tau:=\inf\{t:\,x(t)\not\in D\}$, and
assume (i)--(iii). Then $\tau=\infty$ (a.s.).
\end{lemma}

\begin{proof}  Define $\tau^{k}$ as the first exit time of $x(t)$ from
$D_{k}$. Obviously $\tau^{k}\uparrow\tau$. Therefore to prove the
lemma it suffices to show that for any $k$ and $\delta,T>0$ we have
\begin{equation}
                                               \label{6.4.2.3}
P(\tau^{k}\leq T)\leq P(\xi\not\in
D_{k})+P(V(0,\xi)\geq\log(1/\delta)) +{1\over \delta
V_{k}(T)}\exp\int_{0}^{T}M(t)\,dt.      
\end{equation}

Apply It\^o's formula   to $\gamma (t)V(t,x (t))$ where
$$
\gamma (t):=\exp \Big[-\int_0^t M(s)\,ds-V(0,\xi)\Big],
$$ and     use  assumption ($ii$). Then it follows that for all $t$
$$
\gamma (t)V(t\wedge\tau^k,x (t\wedge\tau^k))I_{\tau^k >0}\leq
\gamma (0)V(0,\xi) +m^k(t),
$$ 
 where $m^k(t)$ is a continuous local martingale starting from 0.
Hence  for any $R>0$
$$ P\{\sup_{t\leq \tau^k}\gamma (t)V(t,x^k(t))I_{\tau^k >0}\geq R\}
\leq \frac{1}{R}E(\gamma (0)V(0,\xi))\leq \frac{1}{R},
$$ 
and this gives \eqref{6.4.2.3} almost immediately. The lemma is proved. 
\end{proof}

In order to state our main results we need one more notion.

\begin{definition}
                                            \label{definition 2.3}
 We say that the pathwise uniqueness holds
for equation \eqref{6.4.2.1} if for any stochastic basis carrying a
$d_{1}$-dimensional Wiener process
$w'(\cdot)$ and a random variable $\xi'$ such that the joint
distribution of
$(w'(\cdot),\xi')$ is the same as that of the given $(w(\cdot),\xi)$,
equation \eqref{6.4.2.1} with $w'(t),\xi'$ in place of $w(t),\xi$ cannot have
more than one solution.
\end{definition}

\begin{theorem}
                                            \label{Theorem 2.4} 
Assume (i)--(iii). Suppose moreover that $b$
and 
$\sigma$ are continuous in $x\in D$ and that for equation \eqref{6.4.2.1} the
pathwise uniqueness holds. Then $x_n(t)$ converges in probability to
a process $x(t)$, uniformly in $t$ in bounded intervals, and 
$x(t)$ is the  unique solution of equation \eqref{6.4.2.1}. Furthermore, $x(t)$
is
$\cF^{w}_{t}\vee\sigma(\xi)$--adapted.
\end{theorem}

\begin{remark}
                                               \label{Remark 2.5} 
 Note that taking $V(t,x):=(|x|^2+1)\exp
(-\int\limits_{0}^{t}M(s)\,ds)$  in the case $D=\Bbb R^d$, 
$D_k:=\{x\in \Bbb R^d: \,|x|<k\}$,  conditions ($i$)--($ii$) can be
restated as follows:
\begin{itemize}
\item $\sup_{|x|<k}\{|b(t,x)|+|\sigma(t,x)|^{2}\} \leq M_k(t)$  for
every $t\geq 0$ and   integer $k\geq1$;
\item $2(x,b(t,x))+\|\sigma (t,x)\|^2\leq M(t)(|x|^2+1)$  for every
$t\geq 0$ and $x\in \Bbb R^d$,
\end{itemize}
 where $\|\alpha\|$ denotes the Hilbert--Schmidt norm for
matrices $\alpha$ and $(x,y)$ is the scalar product of $x,y\in\bR^{d}$.
\end{remark}

We say that the coefficients $b$, $\sigma$ satisfy the monotonicity
condition on $D$ if for every $k$ and $t\geq 0$, $x,\, y \in D_k$ we
have
$$ 2(x-y,b(t,x)-b(t,y))+\|\sigma (t,x)-\sigma (t,y)\|^2\leq
M_k(t)|x-y|^2.
$$

\begin{corollary}[c.f. \cite{6.}]
                                         \label{Corollary 2.6}
 Assume (i)--(iii) and let
the coefficients $b$, $\sigma$ satisfy the monotonicity condition on
$D$. Or in case $D=\Bbb{R}^{d}$ we may assume that the conditions 1
and 2 from Remark 2.5 are satisfied and that the monotonicity
condition is satisfied for $D_{k}=
\{x\in \Bbb R^d: \,|x|<k\}$.
 Assume moreover  that $b$ is continuous  in $x\in D$. Then the
conclusions of Theorem \ref{Theorem 2.4} hold.
\end{corollary}

 \begin{proof} One can easily show that the monotonicity condition
implies the  pathwise uniqueness (see e.g. Krylov \cite{7.}). Hence this
corollary is immediate  from Theorem \ref{Theorem 2.4}.
\end{proof}

In the one-dimensional case (i.e. when $d=1$)  we have the following
result.

\begin{corollary}
                                             \label{Corollary 2.7}
 Let $d=1$.  Assume (i)--(iii) and let $b$ be
continuous in
$x$ in $D$ for any $t$. Assume moreover that for every $k$  and
$t\geq 0,\,\,x\,,y\in  D_k$ we have
$$ (x-y,b(t,x)-b(t,y))\leq M_k(t)|x-y|^2,\quad |\sigma(t,x)-\sigma
(t,y)|^{2}\leq M_k(t) \rho_{k} (|x-y|),
$$ 
 where   $\rho_k\geq0$ is an increasing  function on $\Bbb R_+$ such that
$$
\int_0^\varepsilon 1/\rho_k(r)\,dr=\infty          
$$ for some $\varepsilon> 0$. Then the conclusions of Theorem \ref{Theorem 2.4}
hold.
\end{corollary}

\begin{proof} For any given $k=1,2,...$ we can make  a nonrandom time
change which reduces the general case to the case $M_{k}\equiv1$. In
this case one can see by a straightforward
 modification of the well--known method from Yamada and Watanabe \cite{15.}
(see also Ikeda and Watanabe \cite{4.}) that  the above conditions  imply
the pathwise uniqueness for solutions of equation \eqref{6.4.2.1} until they
leave $D_{k}$. Of course, after this we see that even without time
change we have the pathwise uniqueness for solutions until they leave
$D_{k}$. Since this is true for any $k$ we have the pathwise
uniqueness in $D$, and this is the only thing we need to apply
Theorem \ref{Theorem 2.4}.
\end{proof}

If we are dealing with nondegenerate equations, the continuity
condition on $b$ in Theorem \ref{Theorem 2.4} can be dropped. To state this more
precisely, in addition to the conditions ($i$)--($iii$) let us
introduce the following non-degeneracy  condition on the diffusion
coefficient $\sigma$:

($iv$) For every $k$ the domain $D_k$ is    bounded and
convex,   and 
$$
\sum _{i,j}{(\sigma\sigma^{ *})_{ij}(t,x)}{\lambda}^i {\lambda}^j \geq
\varepsilon_k  M_k(t)\sum_i{|\lambda^i|}^2
$$ for every $t\in [0,k]$, $x\in D_k$, $ \lambda  \in \Bbb R$,
where $\varepsilon_k >0$ are  some constants.

We say that a function $f$ on $\Bbb R_+\times D$ is locally H\"older 
in $x$ in $D$ (with exponent 
$\alpha\in(0,1]$) if    for every $k$  and $t\geq 0$, $x\,, y\in D_k$
$$ |f(t,x)-f(t,y)|^2\leq M_k(t)|x-y|^{2\alpha}.
$$ 
 If $\alpha=1$, then we say that $f$ is  locally Lipschitz in $x$ in
$D$. 

\begin{theorem} 
                                            \label{Theorem 2.8}
  Assume (i)--(iv) and suppose that $\sigma$ 
is locally H\"older in $x$ in $D$  with some exponent
$\alpha\in(0,1]$. In the case $\alpha\not=1$ assume in addition
 that the pathwise uniqueness  holds for equation \eqref{6.4.2.1}.   Then
Euler's approximations $x_n(t)$ converge to a  process $x(t)$ in
probability, uniformly in $t$ in bounded intervals, and
$x(t)$ is the unique solution of equation \eqref{6.4.2.1}. Furthermore, $x(t)$
is
$\cF^{w}_{t}\vee\sigma(\xi)$--adapted.
\end{theorem} 

In the one-dimensional case one can state a condition on pathwise
uniqueness differently. 

\begin{corollary}
                                             \label{Corollary 2.9} 
Let $d=1$ and assume (i)--(iv). Suppose
that $\sigma$ is locally H\"older in $x$ in $D$
 with some exponent $\alpha\in(0,1]$. Assume moreover that for every
$k$
$$ |\sigma(t,x)-\sigma(t,y)|^2\leq M_k(t)(\rho_k (|x-y|)
+|v_k(x)-v_k(y)|)
$$ for every $t\geq 0$, $x,y\in D_k$, where $v_k$ is a real-valued function
of  locally bounded variation and $\rho_k$ is an increasing
continuous function satisfying
$$
\int_0^\varepsilon 1/(r\vee\rho_k(r))\,dr=\infty          
$$ for some $\varepsilon> 0$. Then the conclusions of Theorem \ref{Theorem 2.8} 
hold.
\end{corollary}

\begin{proof} Using the result obtained in Veretennikov \cite{14.}  on
pathwise uniqueness for  stochastic It\^o's equations in one
dimension (which generalizes the  corresponding results in Yamada and
Watanabe \cite{15.}  and in Nakao \cite{10.}), we can repeat the argument from the
proof of Corollary 2.7.
\end{proof}

Finally we present a result on Euler's approximations for the equation
\begin{equation}                                                                      \label{6.21.2}
dx(t)=b(t,x(t))\,dt+\sigma(t,x(t))\,dw(t), \quad x(0)=\xi
\end{equation}
with locally unbounded drift $b=b(t,x)$ and bounded uniformly 
non-degenerate $\sigma=\sigma(t,x)$, 
which is H\"older continuous in $x$, where $\xi$ is and $\cF_0$-measurable random 
vector in $\bR^d$. 

We will be dealing with tamed Euler approximations for \eqref{6.21.2} 
defined as 
\begin{equation}                                                                      \label{6.21.3}
dx_n(t)=b_n(t,x_n(\kappa_n(t)))\,dt+\sigma(t,x_n(\kappa_n(t)))\,dw(t), \quad x_n(0)=\xi, 
\end{equation}
where $b_{n}$ are certain functions. 

To formulate our conditions, for $p,q\in[1,\infty]$ 
we introduce the notation 
$L_{p}=L_{p}(\bR^{d})$ for the usual space of Borel functions on $\bR^{d}$
summable to the power $p$ with norm $\|\cdot\|_{p}$ and use 
$
L_{p, q }(T)=L_{p, q }([0,T]\times\bR^d)$ for 
the space of Borel functions $f=f(t,x)$ on $[0,T]\times\bR^d$ such that 
$$
 \|f\| _{p, q ,T}
=\Big(\int_0^T \|f(t)\| _{p}^{ q }dt\Big)^{1/ q }<\infty
\quad\text{when $ p , q \in[1,\infty)$},    
$$
\begin{equation}                                                                      \label{notation 6.21.5}
\|f\|_{\infty, q ,T}=\Big(\int_0^T\sup_{\bR^d}|f(t,x)|^{ q }\,dt\Big)^{1/ q }<\infty
\quad \text{when $p=\infty$, $ q \in[1,\infty)$}
\end{equation}
and  $\|f\|_{p,\infty,T}=\lim_{ q \to\infty}\|f\|_{p, q ,T}<\infty$ 
when $ q =\infty$, $p\in[1,\infty]$. 
 
\begin{assumption}
                      \label{assumption 7.5.1}
(1) The diffusion coefficient $\sigma$ is  a Borel function on 
$[0,\infty)\times\bR^d$ such that for each $T\in[0,\infty)$ there are constants 
$\varepsilon>0$ , $K<\infty$ and $\alpha\in(0,1)$ such that 
\begin{equation}                                                                          {\label{condition 6.21.1}}
\varepsilon I\leq(\sigma\sigma^{*})(s,x)\leq KI,
\quad
|\sigma(s,x)-\sigma(s,y)|\leq K|x-y|^{\alpha}
\end{equation}
for $s\in[0,T]$ and $x,y\in\bR^d$. 

(2) For each $T>0$ we have $|b|\in L_{2p,2 q ,T}$ for some 
$ q \in(1,\infty)$ and $p\in(\tfrac{d}{\alpha}
,\infty)$,  
  such that
\begin{equation}  
                                             \label{6.27.1}
\frac{d}{p}+\frac{2}{ q }<2.
\end{equation}

For each $T\in(0,\infty)$ we have $b_{n}\to b$ 
in $L_{2p,2 q ,T}$.

(3) For each $T>0$ 
there is a constant $\delta(T)>0$ such that 
\begin{equation}                                                                            \label{condition 6.19.1}
\min_{i:t_{i+1}\leq T}(t^n_{i +1}-t^n_{i })/d_n(T)\geq \delta(T) 
\quad\text{for all $n\geq1$}, 
\end{equation} and for a $\gamma=
\gamma(T)\in(0,( q -1)/ q )$
$$
B(T):= \sup_{n\geq1}d_n^{\gamma/2}(T)
\|b_{n}\|_{\infty,2 q ,T}<\infty.
$$
\end{assumption}

\begin{theorem}                                                                     \label{theorem 6.21.5}                      
Under Assumption
\ref{assumption 7.5.1} suppose that for equation 
\eqref{6.21.2} the pathwise uniqueness holds. 
Then the tamed Euler approximations  converge  in probability, 
uniformly on finite time intervals,  to a continuous 
$\cF^{ w }_t\vee\sigma(\xi)$-adapted process  $x(t)$, which is the unique 
solution of \eqref{6.21.2}.
\end{theorem}

\mysection{Proof of Theorem \ref{Theorem 2.4}}
  \label{section 3}
 
 For every positive integers $k$, $n$ define the stopping
time 
$$
\tau_n^k:=\inf \{t\geq 0:\, x_n(t)\notin D_k\}.
$$
 Then
$$ |b(t, x_n(\kappa_n (t)))|\leq M_k(t),\qquad |\sigma
(t,x_n(\kappa_n (t)))|^{2}\leq M_k(t)
$$ for $t\leq \tau_n^k$, and clearly the family of stochastic
processes
 $\{ x_n^k\,:\,n=1,2,...\}$ defined by
$$ x_n^k(t):=x_n(t\wedge \tau_n^k), 
$$ is  tight in  $C([0,T])$ for  every $k$ and   $T\geq 0$.
We want to deduce from this the  tightness  in  $C([0,T])$ 
of 
\begin{equation}
                                                                  \label{3.1}
\{ (x_n(t))_{t\in
[0,T]}\,:\,n=1,2,...\}.                                         
\end{equation}
Clearly, it suffices to show that  
\begin{equation}
                                                                 \label{3.2}
\lim_{k\to\infty}\limsup_{n\to\infty}P(\tau_n^k \leq
T)=0.                             
\end{equation}

At first fix $k$ and apply Skorokhod's embedding theorem. Then by
virtue of the  tightness  of distributions of $x_n^k(t)$ in
$C([0,T])$  for every $T\geq 0$, we can find a subsequence $n(j)$ and 
a probability space $(\tilde\Omega,\tilde \cF,\tilde P)$, 
carrying the sequences of continuous processes 
$\tilde x_{n(j)}^k$,  $\tilde w_j$,  such that for
every                        positive integer $j$ finite dimensional
distributions of 
$$ 
(\tilde x_{n(j)}^{k},\tilde w_j)\,\,\text
{and}\,\,(x_{n(j)}^{k},w)
$$ 
 coincide, and for any $T<\infty$  for $\tilde P$--almost every
$\tilde \omega\in\tilde\Omega$
\begin{equation}
                                                  \label{3.3}
\sup_{t\leq T}|\tilde x_{n(j)}^k(t)-\tilde x^{k}(t)|\to
0,\,\,\,\,\,\,
\sup_{t\leq T}|\tilde w_j(t)-\tilde w(t)|\to
0,			                                    
\end{equation}
 as $j\to\infty$, where $\tilde x$, $\tilde w$ are some stochastic
processes. Define $\tilde{\tau}^{k}_{n(j)},\tilde{\tau}^{k}$ as the
first exit times from
$D_{k}$ of the processes $\tilde x_{n(j)}^k,\tilde x^{k}$,
respectively. It  follows from (3.3) that
\begin{equation}
                                                  \label{3.4}
\liminf\limits_{j\to\infty}\tilde{\tau}^{k}_{n(j)}\geq\tilde{\tau}^{k}\,\,\,
{\text
(a.s.)}.                                                                     
\end{equation}
 Next define 
$$
\tilde \cF_t^j: =\sigma(\tilde x_{n(j)}^{k}(s),  \tilde
w_j(s):s\leq t),
\,\,\,\,\,\,
\tilde \cF_t : =\sigma(\tilde x^{k}(s),  \tilde w(s):s\leq t).
$$ 
Then it is easy to see that for every $j$ the process $(\tilde
w_j(t),\tilde\cF_t^j)$  and $(\tilde w (t),\tilde\cF_t)$  are
Wiener processes, and for all $t\in [0,\tilde{\tau}^{k}_{n(j)})$
\begin{equation}
                                                  \label{3.5}
\tilde x_{n(j)}^{k}(t)=\tilde x_{n(j)}^{k}(0)+
\int_0^t b(s,\tilde x_{n(j)}^{k}(\kappa_{n(j)}(s)))\,ds+
\int_0^t \sigma(s,\tilde x_{n(j)}^{k}(\kappa_{n(j)}(s)))\,d\tilde
w_j(s),               
\end{equation}
 almost surely. Now we make use of the following lemma which is
just an adaptation of a result of Skorokhod \cite{13.}.

\begin{lemma}
                                                \label{Lemma 3.1} 
Let $f(s,x)$ be a continuous in $x$ and Borel in
$s$ bounded function defined on $\Bbb{R}_{+}\times\Bbb{R}^{d}$. Then
for any $i=1,...,d_{1}$
$$
\int_0^t f(s,\tilde x_{n(j)}^{k}( s ))\,ds \to
\int_0^t f(s,\tilde x^{k}( s ))\,ds, 
$$
$$
\int_0^t f(s,\tilde x_{n(j)}^{k}(\kappa_{n(j)}(s) ))\,ds
\to                      
\int_0^t f(s,\tilde x^{k}( s ))\,ds, 
$$
$$
\int_0^t f(s,\tilde x_{n(j)}^{k}( s ))\,d\tilde w_{j}^{i}(s)\to
\int_0^t f(s,\tilde x^{k}( s ))\,d\tilde w^{i}(s),
$$
\begin{equation}
                                                  \label{3.6} 
\int_0^t f(s,\tilde x_{n(j)}^{k}(\kappa_{n(j)}(s)))\,d\tilde
w_{j}^{i}(s)\to
\int_0^t f(s,\tilde x^{k}( s ))\,d\tilde w^{i}(s)                                                                                  
\end{equation}
 uniformly in $t\in[0,T]$ in probability for any $T<\infty$.
\end{lemma}

Owing to (3.4) and (3.6) we then conclude that for
$t<\tilde{\tau}^{k}$ (a.s.)
$$
\tilde x^{k}(t)=\tilde x^{k}(0)+
\int_0^t b(s,\tilde x^{k}( s ))\,ds+
\int_0^t \sigma(s,\tilde x^{k}(s))\,d\tilde w(s).            
$$

In the proof of estimate \eqref{6.4.2.3} we have used only that $x(t)$
satisfies equation \eqref{6.4.2.1} until it hits $\partial D_{k}$. Therefore
estimate \eqref{6.4.2.3} holds for our
$\tilde{\tau}^{k}$, and since $\tau^{k}_{n}$ have the same
distributions as
$\tilde{\tau}^{k}_{n}$,
$$
\lim_{k\to\infty}\limsup_{j\to\infty}P(\tau_{n(j)}^{k} \leq T)=
\lim_{k\to\infty}\limsup_{j\to\infty}P(\tilde{\tau}_{n(j)}^{k} \leq
T)
$$
$$
\leq\lim_{k\to\infty}P(\tilde{\tau}^{k} \leq T)=0.
$$

The arbitrariness in the choice of the subsequence $n(j)$ allows us to
assert that \eqref{3.2} holds, and thus the family \eqref{3.1} is indeed  tight.  

On our way of applying Lemma 1.1 we now take two subsequences $x_l$,
$x_m$  of the approximations 
$\{x_n\}_{n=1}^{\infty}$. Then obviously $\{(x_l,x_m) \}$ is a tight
family   of processes in
$C([0,T];\Bbb R^{2d})$ for any $T<\infty$.  Again by Skorokhod's
embedding theorem there exist subsequences $l(j)$, $m(j)$, a
probability space $(\hat\Omega,\hat \cF,
\hat P)$,  carrying  sequences of continuous processes 
$\hat x_{l(j)}$, $\bar x_{m(j)}$, $\hat w_j$,  such that for every
positive integer $j$ the finite dimensional distributions of 
$$ (\hat x_{l(j)},\bar x_{m(j)},\hat w_j)\,\,\text
{and}\,\,(x_{l(j)},x_{m(j)},w)
$$  coincide, and for $\hat P$--almost every $\hat
\omega\in\hat\Omega$
$$
\sup_{t\leq T}|\hat x_{l(j)}(t)-\hat x(t)|\to 0,\qquad \sup_{t\leq T}|
\bar x_{l(j)}(t)-\bar x(t)|\to 0,
$$
$$
\sup_{t\leq T}|\hat w_j(t)-\hat w(t)|\to 0,
$$ as $j\to\infty$ for any $T<\infty$,  where $\hat x$, $\bar x$,
$\hat w$ are some stochastic processes.

In the same  way as above we get that for any $k$ the processes
$\hat{x}(t)$ and $\bar{x}(t)$ satisfy equation \eqref{6.4.2.1} on the time
intervals
$[0,\hat{\tau}^{k})$ and $[0,\bar{\tau}^{k})$ respectively with
$\hat{w}$ instead of $w$, where $\hat{\tau}^{k}$ and $\bar{\tau}^{k}$
are defined in an obvious way. Again as above $\hat{\tau}^{k}$,
$\bar{\tau}^{k}\to\infty$, so that actually $\hat{x}(t)$  and
$\bar{x}(t)$ satisfy the corresponding equation on 
$[0,\infty)$. Since the initial condition in both cases is the same
($\hat{x}_{l(j)}(0)=\bar{x}_{m(j)}(0)$ because
$x_{l}(0)=x_{m}(0)=\xi$) and since the joint distribution of the
initial value and $\hat{w}$ coincides with the distribution of
$\xi,w$, by the pathwise uniqueness we conclude that
$\hat{x}(t)=\bar{x}(t)$ for all $t$ (a.s.).  Hence, by applying Lemma
1.1 we finish the proof of Theorem \ref{Theorem 2.4}.

\mysection{An estimate of densities for Euler's approximations}

 In the case when the coefficients of equation \eqref{6.4.2.1} are not
supposed to be continuous, in order to apply the above scheme we need
a counterpart of Lemma \ref{Lemma 3.1} for measurable $f$. The proof of the
corresponding assertion is based on an estimate on the densities of
distribution of the Euler approximation
$x_n(t)$. Since such  estimates can be applied in other situations,
the result we prove
 below is stronger than we actually need in the proof of Theorem \ref{Theorem 2.8}.

First of all we need the following lemma.

\begin{lemma}
                                                 \label{Lemma 4.1}
 Let $K,t,\varepsilon>0,\alpha\in(0,1]$ be fixed
numbers, and
 let $a(x)$ be a $d\times d$ matrix-valued function such that
$KtI\geq a=a^{*}\geq \varepsilon tI$,  where $I$ is the $d\times d$
unit matrix. Also let
  $g(x)$ be a real-valued function such that $|g(x)-g(y)|\leq
K|x-y|^{\alpha}$ for all $x,y$. Let $\xi$ and $\eta$ be independent
$d$-dimensional Gaussian vectors with  zero means.  Assume $\xi\sim
{\cN} (0,I)$ and $\kappa_{i}\leq K\kappa_{j}$
for $i,j=1,...,d$, where the $\kappa_{i}$'s are
the eigenvalues of the covariance matrix of $\eta$.
 Define an operator $T^*$ by the formula
$T^{*}f(y)=Ef(y+\sqrt{a(y)}\xi)$ and let $T$ be  the conjugate for
$T^{*}$ in $L_2$--sense. Then for any $i,j=1,...,d$, $x\in\Bbb
R^{d}$, $p\in[1,\infty]$, and bounded Borel
$f$
\begin{equation}
                                                      \label{4.1}
 \Big|g(x)E[D_{ij}T f](x+\eta)- E[D_{ij}T(gf)](x+\eta)\Big|\leq
Nt^{-d/(2p)-1+\alpha/2}\|f\|_{p},                                                   
\end{equation}

\begin{equation}
                                                      \label{4.2}
\Big\{\int_{\Bbb R^{d}}\big|g(x)E[D_{ij}T f](x+\eta)- E[D_{ij}T(gf)](x+\eta)\big|^{p}\,dx\Big\}^{1/p}\leq
Nt^{-1+\alpha/2}\|f\|_{p},                                                  
\end{equation}
 where the constants $N$ depend only on $K,\varepsilon,d,p$.
\end{lemma}
\begin{proof} First observe that 
\begin{equation}
                                            \label{6.11.1}
 Tf(x)
=\int_{\Bbb R^{d}}(2\pi\det a(y))^{-d/2}f(y) 
\exp\{-(a^{-1}(y)(y-x),y-x)/2\} \,dy,
\end{equation}
$$
E[D_{ij}T f](x+\eta)=D_{ij}E T f (x+\eta),
\quad E[D_{ij}T(gf)](x+\eta)=D_{ij}E T(gf) (x+\eta),
$$
$$ 
E(2\pi\det a)^{-d/2}\exp\{-(a^{-1}(y-x-\eta),y-x-\eta)/2\}
$$
$$ 
=(2\pi\det(a+a_{1}))^{-d/2}\exp\{-((a+a_{1})^{-1}(y-x ),y-x )/2\}
=:p_{a}(x,y),
$$ where $a_{1}$ is the covariance matrix of $\eta$. Note 
also that 
$$
KtI+a_{1}\geq a(y)+a_{1}\geq \varepsilon tI+a_{1}
\geq(\varepsilon/K)(KtI+a_{1}),
$$
$$
(KtI+a_{1})^{-1}\leq (a(y)+a_{1})^{-1}
\leq (\varepsilon tI+a_{1})^{-1}
\leq (K/\varepsilon)(KtI+a_{1})^{-1},
$$
$$ \det(a(y)+a_{1})\geq\det(\varepsilon tI+a_{1})\geq (\varepsilon/K)^{d}\det(KtI+a_{1}).
$$
It follows that $p_{a}(x,y)  \leq r(x-y)$,
where
$$
r(z):=
( K/\varepsilon)^{d^{2}/2}\big(2\pi\det(KtI+a_{1})\big)^{-d/2}
\exp\big\{-\big((KtI+a_{1})^{-1}z,z\big)/2
\big\}.
$$

Next, let
$A(y)=(a(y)+a_{1})^{-1}$,  then
$$ 
ED_{ij}Tf(x+\eta)
$$
$$
=
\int_{\Bbb R^{d}} f(y) [(A(y)(y-x))_{i}(A(y)(y-x))_{j}- A_{ij}(y)]
p_{a(y)}(x,y)\,dy.
$$
This allows us to deal with
$$
I_{\lambda}(x):=
\lambda^{i}\lambda^{j}\Big(g(x)E[D_{ij}T f](x+\eta)- E[D_{ij}T(gf)](x+\eta)\Big),
$$
where $\lambda$ is a fixed vector in $\bR^{d}$. By the above
$$ 
I_{\lambda}(x)=
 \int_{\Bbb R^{d}}[g(x)-g(y)] f(y)
\big[(\lambda,A(y)(y-x))^{2}-
\lambda^{i}\lambda^{j} A_{ij}(y)\big] p_{a(y)}(x,y)\,dy .
$$
By ordering the eigenvalues of $a_{1}$ as $\kappa_{1}
\leq...\leq\kappa_{d}$ we have
$$
 0\leq \lambda^{i}\lambda^{j} A_{ij}(y)
\leq(\varepsilon t+\kappa_{1})^{-1}|\lambda|^{2},
$$
$$
|(\lambda,A(y)z)|\leq(\lambda,A(y)\lambda)^{1/2}(z,A(y)z)^{1/2}
$$
$$
\leq
(\varepsilon t+\kappa_{1})^{-1/2}|\lambda|(K/\varepsilon)^{1/2}
(z,(K tI+a_{1})^{-1}z)^{1/2}.
$$

Using this and making the change of variables $x-y=\sqrt{KtI+a_{1}}z$
we find that $|I_{\lambda}(x)|$ is less than or equal to
$$
 N(\varepsilon t+\kappa_{1})^{-1}|\lambda|^{2}
  \int_{\Bbb R^{d}}|f(y)\|x-y|^{\alpha}
 \Big[\big(x-y,(K tI+a_{1})^{-1}(x-y)\big)+1\Big]r(x-y)\,dy
$$
$$
= N (\varepsilon t+\kappa_{1})^{-1}|\lambda|^{2}
\int_{\Bbb R^{d}}
|\sqrt{KtI+a_{1}}z|^{\alpha }\big|f(x-
\sqrt{KtI+a_{1}}z)\big[|z|^{2}+1\big]e^{-|z|^{2}/2}\,dz
$$
$$
\leq N(\varepsilon t+\kappa_{1})^{\alpha/2 -1}|\lambda|^{2}
\int_{\Bbb R^{d}}\big|f(x-\sqrt{KtI+a_{1}}z)
\big\|z|^{\alpha}[|z|^{2}+1]e^{-|z|^{2}/2}\,dz.
$$
The arbitrariness of $\lambda$ implies that
$$
\Big|g(x)E[D_{ij}T f](x+\eta)- E[D_{ij}T(gf)](x+\eta)\Big|
$$
$$
\leq N(\varepsilon t+\kappa_{1})^{\alpha/2-1 } 
\int_{\Bbb R^{d}}\big|f(x-\sqrt{KtI+a_{1}}z)\big\|z|^{\alpha}
[|z|^{2}+1]e^{-|z|^{2}/2}\,dz.
$$
$$
\leq N(\varepsilon t+\kappa_{1})^{\alpha/2 -d/(2p)-1} \|f\|_{p}
\leq Nt^{-d/(2p)-1+\alpha/2}\|f\|_{p}.
$$ 
Here we have used the H\"older inequality. To prove \eqref{4.2} we apply
instead the Minkowski inequality. The lemma is proved.
\end{proof}

We will apply Lemma 4.1 to prove some estimates for distributions of
the process $x_{n}(t)$ defined as
\begin{equation}                                                                          \label{6.17.1}
x_{n}(t)=x_{0}+\int_{0}^{t}\sigma(s,x_{n}(\kappa_{n}(s)))\,dw(s),
\end{equation} 
where $x_0\in\Bbb R^d$ is non random and 
$\sigma:\Bbb R_+\times \Bbb R^d :\to\Bbb R^{d\times d_{1}}$ is Borel
measurable and satisfies the  condition 
\begin{equation}                                                      
                         \label{condition 6.17.2} 
\varepsilon I\leq(\sigma\sigma^{*})(s,x)\leq KI,
\quad
|\sigma(s,x)-\sigma(s,y)|\leq K|x-y|^{\alpha}
\end{equation}
for some constants $\alpha\in(0,1)$, $K,\varepsilon >0$ and all
$x,y\in \Bbb R^{d}$,
$s>0$.

Before stating the main result of this section we introduce some
notation. For fixed $n$ and $t>0$ a very cumbersome expression  can
be found explicitly in an obvious way for the distribution density
$p_{n}(t,x)$ of   $x_{n}(t)$. We do not know if it is possible to
estimate the density analyzing this expression, but at least it shows
that  the density is bounded on $[\delta,\delta^{-1}]\times \Bbb
R^{d}$ for any $\delta>0$. We denote by $m_{n}(t)$ the supremum of
$p_{n}(t,x)$  over $x\in \Bbb R^{d}$. The function $m_{n}(t)$ is
bounded on 
$[\delta,\delta^{-1}]$ for any $\delta>0$ and any $n$.

\begin{theorem}
                                                                            \label{Theorem 4.2}
  (a) There exists a constant $N_{0}$ depending
only on $d,\alpha, K,\varepsilon,q$ such that if $1\leq q<\frac
{d}{d-\alpha}$, then for all $t>0$, $n=1,2,3,...$
\begin{equation}
                                                    \label{4.3}
\big(\int_{\Bbb R^{d}}p_{n}^{q}(t,x)\,dx\big)^{1/q}
\leq  N_{0} (t^{- d/(2p)}+1)\ \ \
(p=q/(q-1)\,).                                    
\end{equation}

(b) If the partitions
 $\{0=t_0^n< t_1^n<...\}$ satisfy the additional condition
$\kappa_n(s)\geq\varepsilon s$ for all $n$ and $s\geq t_{1}^{n}$, then  there
exists a constant $N_{0}$ depending only on $d,\alpha, K,\varepsilon$
such that
\begin{equation}
                                                    \label{4.4}
 m_{n}(t)\leq
N_{0}(t^{-d/2}+1)                                               
\end{equation}
for any   $t>0$,  $n=1,2,3,...$
 and   \eqref{4.3}  holds for any  $t>0$,  $n=1,2,3,...$.
\end{theorem}

\begin{proof} The last assertion in ($b$) is true since 
$p_{n}^{q}\leq p_{n}(m_{n})^{q-1}$ and
$\int p_{n}\,dx=1$. To prove ($a$) for $0\leq s\leq t<\infty$ and
 bounded measurable $f(x)$ let
$$ T_{s,t}^{*}f(y):=Ef(y+\int_{s}^{t}\sigma(r,y)\,dw(r)),
$$ and  let the operator $T_{s,t}$ be conjugate to $T_{s,t}^{*}$ in
$L_{2}$-sense. The expression $T_{s,t}f(x)$ can be written as an
integral with respect to a  Gaussian-like density, and from this
formula it is not hard to see that for any $t$ the function
$T_{s,t}f(x)$ is infinitely differentiable
for $s<t$  and
\begin{equation}
                                                    \label{4.5}
 {\partial\over\partial s}T_{s,t}f(x)=-D_{ij}T_{s,t}a^{ij}(s,\cdot)f(\cdot)(x),                
\end{equation}
 where $a_{ij}:=\frac {1}{2} (\sigma\sigma^{ *})_{ij}$. For the sake
of simplicity of notations we drop the subscripts $n$, and from  \eqref{4.5}  
by  the Newton-Leibnitz and It\^o's formulas for any $r\in[0,t]$ we
obtain
$$ Ef(x(t))=\int_{r}^{t}{d\over
ds}ET_{s,t}f(x(s))\,ds+ET_{r,t}f(x(r))= ET_{r,t}f(x(r))
$$
$$
+\int_{r}^{t}E\Big[a^{ij}(s,x(\kappa(s)))D_{ij}
T_{s,t}f(x(s)) -D_{ij}T_{s,t}a^{ij}(s,\cdot)f(\cdot)(x(s))\Big]\,ds.
$$ We take the conditional expectations given $x(\kappa(s))$, and after
denoting 
$$
\eta(s,x)=\int_{\kappa(s)}^{s}\sigma(r,x)\,dw(r)
$$
 we get
\begin{equation}
                                                    \label{4.6}
Ef(x(t))=ET_{r,t}f(x(r))+\int_{r}^{t}EH(s,t,x(\kappa(s)))\,ds,             
\end{equation}
 where
$$
 H(s,t,x)=a_{ij}(s,x)E[D_{ij}T_{s,t}f](x+\eta(s,x))
$$
$$
-E[D_{ij}T_{s,t}a_{ij}(s,\cdot)f(\cdot)](x+\eta(s,x)). 
$$ 
Note that by Lemma 4.1
$$ 
|H(s,t,x)|\leq N(t-s)^{-d/(2p)-1+\alpha/2}\|f\|_{p},
$$
\begin{equation}
                                                    \label{4.7}
\int_{\Bbb{R}^{d}}|H(s,t,x)|\,dx\leq
N(t-s)^{-1+\alpha/2}\|f\|_{1}.                    
\end{equation}
This and  \eqref{4.6}  with $r=0$ give us 
 \eqref{4.3} for $p>d/\alpha$ 
 and for $t\in(0,T]$ with a constant $N_0$ depending only on 
$d,\alpha, K,\varepsilon,q$ and $T$. 
 Indeed (cf.~\eqref{6.11.1}),
$$ T_{0,t}f(x_{0})\leq Nt^{-d/2}\int_{R^{d}}f(y)
\exp\{-{1\over Nt}(x-y)^{2}\}\,dy\leq Nt^{-d/(2p)}\|f\|_{p},
$$
$$
\int_{0}^{t}(t-s)^{-d/(2p)-1+\alpha/2}\,ds=Nt^{-d/(2p)+\alpha/2}.
$$

To prove   \eqref{4.3}  and  \eqref{4.4} with a constant 
$N_0$ independent of $T$   
we need a longer argument. Fix a $T\in(0,\infty)$,
and define $\gamma_{T}$ as the smallest number $\gamma$ such that
$m(s)\leq\gamma(s^{-d/2}+1)$ for all $s\in(0,T]$. Introduction of
such objects as $\gamma_{T}$ is  rather common in the theory of PDE.
In probability theory they were used for instance in
Stroock--Varadhan    \cite{12.}  for the same purposes. Such a number
$\gamma_{T}$ does exist since $m(t)$ is bounded on $[t_1^n,T]$ and
$m(t)\leq N(d,K,\varepsilon)t^{-d/2}$ for $t\in(0,t_1^n)$ as follows
from the explicit formula for the Gaussian density of
$x(t)=x_0+\int_{0}^{t}\sigma(s,x_{0})\,dw(s)$. We want to estimate
$\gamma_{T}$. We use  \eqref{4.6} 
with $r=t^{n}_{1}$ and $t\in[t_1^n,T] $ and observe
that the first term on the right can be easily estimated
if we take into account \eqref{6.11.1} and use that the 
convolution of Gaussian densities is again Gaussian.
We also use
 \eqref{4.7}  and the inequality $\kappa(s)\geq
\varepsilon s$ ($s\geq t^{n}_{1}$) and  we obtain for $t\in[t_1^n,T] $
$$
 Ef(x(t))\leq Nt^{-d/2}\|f\|_{1} + 
\int_{t^{n}_{1}}^{t}\Big[\gamma_{T}\Big({1\over\kappa^{d/2}(s)}+1\Big) 
\|H(s,t,\cdot)\|_{1}\Big]\wedge\sup_{x}|H(s,t,x)| \,ds
$$ 
$$
\leq\Big\{Nt^{-d/2}+
N\int_{t^{n}_{1}}^{t}\Big[\gamma_{T}\Big({1\over\kappa^{d/2}(s)}+1\Big) 
{1\over(t-s)^{1-\alpha/2}}\Big]\wedge{1\over(t-s)^{d/2+1-\alpha/2}}
 \,ds\Big\}\|f\|_{1},
$$
\begin{equation}
\label{4.8}
 m(t)\leq Nt^{-d/2}+
N\int_{0}^{t}\Big[\gamma_{T}\Big({1\over s^{d/2}}+1\Big) 
{1\over(t-s)^{1-\alpha/2}}\Big]\wedge{1\over(t-s)^{d/2+1-\alpha/2}}
\,ds.                                                                       
\end{equation} 
  Next, as is easy to see   after the substitution
$s=u\gamma_T^{-2/d} $,
$$
\int_{0}^{t}{\gamma_{T}\over 
 (t-s)^{1-\alpha/2}}\wedge{1\over(t-s)^{d/2+1-\alpha/2}} \,ds=
\int_{0}^{t}{\gamma_{T}\over 
 s^{1-\alpha/2}}\wedge{1\over s^{d/2+1-\alpha/2}} \,ds
$$
$$
=\gamma_T^{1-\alpha/d}\int_{0}^{t\gamma^{2/d}_{T}}{1\over 
 u^{1-\alpha/2}}\wedge{1\over u^{d/2+1-\alpha/2}} \,du\leq
N\gamma_T^{1-\alpha/d}.
$$ Upon setting  $u=t\gamma_{T}^{2/d}(1+\gamma_{T}^{2/d})^{-1}$, we
also have
$$
\int_{0}^{t}{\gamma_{T}  \over s  ^{d/2} 
(t-s)^{1-\alpha/2}}\wedge{1\over(t-s)^{d/2+1-\alpha/2}} \,ds
\leq\int_{0}^{u}{1\over(t-s)^{d/2+1-\alpha/2}}\,ds
$$
$$
+\int_{u}^{t}{\gamma_{T}  \over s  ^{d/2}  (t-s)^{1-\alpha/2}}
\,ds\leq {2\over (d-\alpha)(t-u)^{d/2-\alpha/2}}+
\gamma_{T}u^{-d/2}{2\over\alpha}(t-u)^{\alpha/2}
$$
$$ 
=Nt^{-(d-\alpha)/2}(1+\gamma_{T}^{2/d})^{(d-\alpha)/2}
\leq N(1+\gamma_{T}^{1-\alpha/d}) (t^{-d/2}+1).
$$ 
Thus from \eqref{4.8}   for $t\in[t_1^n,T]$ we conclude
\begin{equation}                                                         \label{6.14.1}
 m(t)\leq
N(1+\gamma_{T}^{1-\alpha/d})(t^{-d/2}+1).                            
\end{equation} 
As we observed above this estimate is also true for 
$t\in(0,t_1^n]$. By definition of $\gamma_{T}$ estimate \eqref{6.14.1} means
that
$$
\gamma_{T}\leq N(1+\gamma_{T}^{1-\alpha/d}).
$$ We emphasize that the last constant $N$, as well as all constants
called $N$ in the above proof of  \eqref{4.4}, depends only on 
$d,\alpha, K,\varepsilon$. This implies the desired estimate of
$\gamma_{T}$, and it remains only to notice that the estimate
 is independent of $T$.   We can see in the same way that the constant $N_0$ 
 in the estimate \eqref{4.3} can be taken to be the same for all $t>0$. 
The theorem is proved.
\end{proof}

\begin{corollary}
                                               \label{Corollary 4.3}
 Assume the conditions of Theorem \ref{Theorem 2.8}. Let
$x_n(t)$ be the Euler approximation defined by \eqref{6.4.2} and let
$\tau_{n}^{k}$ be the first exit time of $x_{n}(t)$ from $D_{k}$.
Then for every $t>0$ the measure $P(x_{n}(t)\in\Gamma,
t<\tau_{n}^{k})$ has a density $p_{n}^{k}(t,x)$, and for any 
$0<t_0<T<\infty$,
$1\leq q <\frac {d}{d-\alpha}$ and $k=1,2,...$ we have
\begin{equation}                                                                          \label{6.14.2}
\sup_n\sup_{t\in[t_0,T]}\int_{\Bbb{R}^{d}}[p_{n}^{k}(t,x)]^{q}\,dx<\infty.
\end{equation}
\end{corollary}

Proof. By using a  nonrandom time change  we easily reduce the
general case to the one with $M_{k}(t)\equiv1$. Next we observe that
$$ P(x_{n}(t)\in\Gamma,t<\tau_{n}^{k})\leq P(x_{n}^{k}(t)\in\Gamma),
$$ where $x_{n}^{k}(t)$ are Euler's approximations for equation \eqref{6.4.2.1}
with coefficients $\sigma,b$ changed arbitrarily outside $D_{k}$. 
After this an application of the Girsanov  theorem allows us to take
$b\equiv0$. Finally we get our assertion  from  \eqref{4.3}  if we notice the
obvious relation between Euler's approximations for fixed initial
value and for  random one.
\qed

\begin{remark}
                                            \label{Remark 4.4} 
One knows from Fabes and Kenig \cite{2.} and Safonov
 \cite{11.}  that none  
of the estimates   \eqref{4.3}, \eqref{4.4}  
and \eqref{6.14.2} remains valid
if  the   H\"older continuity of $\sigma$ in $x$ is replaced  by the
assumption of uniform continuity of $\sigma$ in $(t,x)$, 
\end{remark}

\begin{remark}
                     \label{remark 6.28.1}
We derived Theorem \ref{Theorem 4.2}
for approximations starting at time zero
at  a fixed point. Obviously, the approximations can start at any $t ^n _{k}$ and then we get
a ``conditional'' estimate
\begin{equation}
                              \label{6.28.2}
 E\{f(x_{n}(t)\mid \cF_{t_{k}^{n}}\}
\leq N_{0}((t-t ^n _{k})^{-d/(2p)}+1)\|f\|_{p},
\end{equation}
with the same $N_{0}$ as in \eqref{4.3},
whenever $p>d/\alpha$, $k=0,1,2,...$,
$n=1,2,...$,
$t>t ^n _{k}$, and $f$ is Borel.

\end{remark}

\begin{theorem}
                      \label{theorem 6.28.1}
 Let
$x_n(t)$ be the Euler approximation defined by \eqref{6.17.1}, $ q \geq1$, $p>d/\alpha$.
Assume that condition \eqref{condition 6.17.2} is satisfied
and \eqref{condition 6.19.1}, \eqref{6.27.1} hold.
Then for any 
$T\in[0,\infty)$ and 
Borel functions  $f$  on $\bR_{+}\times\bR^d$ we have 
\begin{equation}                                                                             \label{6.18.2}                                                                      
E\exp\Big(\int_0^Tf(r,x_n(\kappa_n(r)))dr\Big)\leq
2\exp\Big(N\big(
\|f \|_{p, q ,T}^{ q }+d_{n}^{  q -1 }(T)\|f \|_{\infty, q ,T}^{ q }\big)\Big),
\end{equation}
where  $N$ depends only   on $d$, $\alpha$, $K$, 
 $\varepsilon$, $p$, $ q $, 
$\delta(T)$, $T$.
\end{theorem}

\begin{proof} We may assume that $f\geq0$ and $f$ is bounded.
Then fix $T$, $n$ and for $t\leq T$ introduce
$$
\psi(r)=f(r,x_n(\kappa_n(r))),\quad
\phi(t)=\int_{t}^{T}\psi(r)\,dr,\quad
\Phi(t)=\esssup E\big\{e^{\phi(t)}\mid \cF_{\kappa_{n}(t)}\big\}.
$$
Observe that
$$
e^{\phi(t)}=1+\int_{t}^{T}\psi(r)e^{\phi(r)}\,dr.
$$
Hence, by taking into account that 
$\psi(r)$ is $\cF_{\kappa_n(r)}$-measurable,  
we get
$$
E\big\{e^{\phi(t)}\mid \cF_{\kappa_{n}(t)}\big\}
=1+\int_{t}^{T}E\{\psi(r)e^{\phi(r)}\mid\cF_{\kappa_{n}(t)} \}\,dr
$$
$$
=1+\int_{t}^{T}E\{\psi(r)E\{e^{\phi(r)}\mid \cF_{\kappa_{n}(r)}\}\mid\cF_{\kappa_{n}(t)}\}\,dr
$$
$$
\leq 1+\int_{t}^{T}\Phi(r)E\{\psi(r)\mid\cF_{\kappa_{n}(t)}\}\,dr 
$$ 
\begin{equation}
                             \label{7.8.1}
\leq 1+  \int_{t}^{\bar\kappa_{n}(t)  \wedge T }\Phi(r) \psi(r)
 \,dr
+ \int_{\bar\kappa_{n}(t) \wedge T }^{ T } \Phi(r)E\big\{\psi(r)
\mid \cF_{\kappa_{n}(t)}\big\}\,dr,
\end{equation}
 where $\bar\kappa_{n}(r)$ is defined to be equal to
  $t^n_{k+ 1} $  if $\kappa_{n}(r)=t ^n _{k}$.  
 Note that for any Borel $h=h(r)\geq0$ and $t\leq r\leq T$
by H\"older's inequality 
(observe that, if $r<t^{n}_{1}$, $x_{n}(\kappa_{n}(r))$ may not have density and this is why we use $\sup$
in notation \eqref{notation 6.21.5})
\begin{equation}
                                 \label{6.28.4}
\int_{r}^{\bar\kappa_{n}(r)}h(s) \psi(s)
 \,ds\leq (\bar\kappa_{n}(r)-r)^{ ( q -1)/ q }
\Big(\int_{r}^{\bar\kappa_{n}(r)}h^{ q }(s) 
\|f(s,\cdot)\|_{\infty}^{ q }
 \,ds\Big)^{1/ q }.
\end{equation}
 Also,
 in light of Remark \ref{remark 6.28.1}
and the fact that 
$\kappa_{n}(r)-\kappa_{n}(t)\geq \delta(T)(r-t) /2 $
for $r\geq \bar\kappa{ _n}(t)$, the last term in \eqref{7.8.1} is dominated by
$$
 N_{0} 2^{d/(2p) }\delta^{-d/ (2p)}(T)\int_{t}^{T}\Phi(r)
\|f(r,\cdot)\|_{p}\big((r-t)^{-d/(2p)}+1\big)\,dr
$$
$$
\leq N \Big(\int_{t}^{T}\Phi^{ q  }(r)
\|f(r,\cdot)\|_{p}^{ q }\,dr\Big)^{1/ q }.
$$    
By using this and \eqref{6.28.4}  we conclude
$$
\Phi(t)\leq 1+N \Big(\int_{t}^{T}\Phi^{ q  }(r)\xi(r)
\,dr\Big)^{1/ q },\quad
\Phi^{ q }(t)\leq 2+N  \int_{t}^{T}\Phi^{ q  }(r)\xi(r)
\,dr ,
$$
where $\xi(r)=\|f(r,\cdot)\|_{p}^{ q }+d_{n}^{  q -1} (T)\|f(r,\cdot)\|_{\infty}^{ q }$. Gronwall's
inequality yields
$$
\Phi^{ q }(t)\leq 2\exp\Big(N \int_{t}^{T}\xi(r)\,dr\Big),
$$
which for $t=0$ implies \eqref{6.18.2} and proves the theorem.
\end{proof}

For fixed 
$T>0$ and $n\geq1$ we denote by $\gamma_n(T)$ the Girsanov exponent 
$$
\gamma_n(T)=\exp\Big(-\int_0^T(\sigma^{-1}b_n)(s,x_n(\kappa_n(s)))dw(s)
-\tfrac{1}{2}\int_0^T|(\sigma^{-1} b_n)(s,x_n(\kappa_n(s)))|^2ds\Big), 
$$
where $x_n(t)$ is the Euler approximation defined by \eqref{6.17.1} 
and $\sigma^{-1}$ is the right inverse of $\sigma$
($\sigma^{-1}=\sigma^{*}(\sigma\sigma^{*})^{-1}$).  Let $\tilde P$ denote  
the probability measure defined by $d\tilde P/dP=\gamma_n(T)$ and use the notation 
$\tilde E$ for the expectation under $\tilde P$.

\begin{proposition}                                                                                            \label{proposition 6.20.1}

Let
$x_n(t)$ be the Euler approximation defined by \eqref{6.21.3}, $ q \geq1$, $p>d/\alpha$.
Assume that conditions \eqref{condition 6.17.2},    
  \eqref{condition 6.19.1}, \eqref{6.27.1} hold.
 Then for any  
$\rho\in\bR$
$$
\tilde E\gamma^{\rho}_n(T)
\leq 2\exp
\Big(
N 
\big(
\|b_n \|_{2p,2 q ,T}^{2 q }+d_{n}^{  q -1} (T)\|b_n \|_{\infty,2 q ,T}^{2 q }
\big)
\Big),
$$
  where   the constant   
$N$    
depends only on $ q $, $p$, $T$, 
$d$, $K$, $\alpha$, $\varepsilon$, $\delta(T)$ and $\rho$.  
\end{proposition}

\begin{proof} We may assume that $x_n(0)=\xi$ is non random.
Notice that 
$$
dx_n(t)=\sigma(t,x_n(\kappa_n(t)))d\tilde w(t)
$$
with 
$$
\tilde w(t)=\int_0^t(\sigma^{-1}b_n)(s,x_n(\kappa_n(s))ds+w(t),
\quad t\in[0,T], 
$$
which is a Wiener process under $\tilde P$ by Girsanov's theorem. Thus 
setting $h_s=(\sigma^{-1}b_n)(s,x_n(\kappa_n(s))$, by simple calculations 
we obtain
$$
\tilde E\gamma^{\rho}_n(T)
=\tilde EH\exp\Big(-\rho\int_0^Th_sd\tilde w_s-\rho^2\int_0^T|h_s|^2ds\Big)
$$
with 
$$
H=\exp\Big((\tfrac{\rho}{2}+\rho^2)\int_0^T|h_s|^2ds\Big). 
$$
Hence by Cauchy-Bunjakovski-Schwarz inequality and using 
Theorem \ref{theorem 6.28.1} we obtain 
$$
\tilde E\gamma^{\rho}_n(T)\leq (\tilde EH^2)^{1/2}
\leq \Big(\tilde E\exp\Big(|\rho+2\rho^2|\int_0^T|b(s,x_n(\kappa_n(s)))|^2ds\Big)\Big)^{1/2}
$$
$$
\leq \sqrt{2}\exp\Big(
N 
\big(
\|b_n \|_{2p,2 q ,T}^{2 q }+d_{n}^{  q -1 }(T)\|b_n \|_{\infty,2 q ,T}^{2 q }
\big)
\Big)
$$
with $N$ depending only on $ q $, $p$, $T$, $d$, $K$, $\alpha$, 
$\varepsilon$, $\delta(T)$ and $\rho$.
\end{proof}
 
\begin{remark}
                             \label{remark 6.30.1}
Observe that, if 
 Assumption 
\ref{assumption 7.5.1} is  satisfied, then for any $\rho$ and $T$
the sequence $\tilde E \gamma^{\rho}_n(T) $ is bounded.
\end{remark}

\begin{proposition}                                                                   \label{proposition 6.21.1}
Suppose that Assumption 
\ref{assumption 7.5.1} is  satisfied and let
$x_n(t)$ be the Euler approximation defined by \eqref{6.21.3}.
Then for any   $\gamma<( q -1)/ q $, $T$, $n$ and bounded Borel functions 
$f\geq0$ on $[0,\infty)\times\bR^d$ we have   
\begin{equation}                                                                        \label{estimate 6.22.2}
E\int_0^Tf(t,x_n(\kappa_n(t)))\,dt\leq N(\|f\|_{p, q ,T}+d_n^{\gamma}(T)\|f\|_{\infty, q ,T})
\end{equation}
with a constant $N$ depending only on $ q $, $p$, $\gamma$, $T$, 
$d$, $K$, $\alpha$, $\delta(T)$,
$B(T)$ and $\varepsilon$. 
\end{proposition}
\begin{proof} First assume $b_n=0$. Clearly, 
$$
E\int_0^T f(t, x_n(\kappa_n(t))dt=E\int_0^{t_1^n \wedge T  } f(t, x_n(\kappa_n(t))dt
+
E\int_{t_1^n \wedge T }^{T }  f(t, x_n(\kappa_n(t))dt,
$$
where the first term can estimated from above by 
$ d_n^{( q -1)/ q }(T)\|f\|_{\infty, q ,T}$. For the second term 
by Theorem \ref{Theorem 4.2} and taking into account that 
$\kappa_n(t)\geq \delta(T)t/2$ for 
$t\geq t^n_1$, we have 
$$ 
E\int_{t_1^n \wedge T }^{T } f(t, x_n(\kappa_n(t))dt
\leq N_0\int_{  t_1^n\wedge T }^{T }(\kappa_n^{-d/2p}(t)+1)\|f(t)\|_{p}^p\,dt
$$
$$
\leq 2^{d/(2p)}\delta^{-d/(2p)}(T)N_0\int_0^T(t^{-d/2p}+1)\|f(t)\|_{p}^p\,dt. 
$$
Hence by H\"older's inequality we get \eqref{estimate 6.22.2}. In the general case 
we use Girsanov's theorem, H\"older's inequality and 
Proposition \ref{proposition 6.20.1} to get 
$$
E\int_0^Tf(t,x_n(\kappa_n(t)))dt
=(\tilde E\gamma_{n}^{-\rho/(\rho-1)}(T))^{1/\rho}
\Big(\tilde E\Big(\int_0^Tf(t,x_n(\kappa_n(t)))dt\Big)^{\rho}\Big)^{1/\rho}
$$
$$
\leq N\Big(\tilde E\int_0^Tf^{\rho}(t,x_n(\kappa_n(t)))\,dt\Big)^{1\rho}
\leq  N'(\|f\|_{\rho p',\rho q ',T}+d_n^{( q '-1)/( q '\rho)}(T)\|f\|_{\infty,\rho q ',T}) 
$$
for any $\rho>1$, $ q '>1$,  $p'>\tfrac{d}{2}\tfrac{ q '}{ q '-1}\vee\frac{d}{\alpha}$ 
which proves the proposition. 
\end{proof} 

\begin{proposition}  
                      \label{proposition 6.24.2}
                      
Under the assumptions of Proposition 
\ref{proposition 6.21.1} suppose that the processes
  $x_n(t)$ converge  to a 
process 
$ x(t) $ in probability, uniformly in $t$ in bounded intervals. 
Then for nonnegative Borel functions $h$ on $[0,\infty)\times\bR^d$ for each $T$ 
we have 
\begin{equation}                                                                        \label{estimate 6.24.1}
E\int_0^Th(t,x(t))\,dt\leq N\|h\|_{p, q ,T}
\end{equation}
for $ q >1$, $p>(\tfrac{d}{2}\tfrac{ q }{ q -1})\vee\tfrac{d}{\alpha}$ 
with a constant depending on $ q $, $p$, $T$, 
$\delta(T)$,
$d$, $K$, $\alpha$ and $\varepsilon$. 
\end{proposition}
\begin{proof}
Letting  $n\to\infty$ in \eqref{estimate 6.22.2} we get \eqref{estimate 6.24.1} when $h$ is a 
continuous function with compact support. Hence the general case follows. 
\end{proof}

\begin{proposition}                                                                         \label{proposition 6.24.6}

Under the assumptions of Proposition 
\ref{proposition 6.21.1}
for each $T>0$ the sequence of processes 
 $x_n(t)$  is  tight 
 in 
$C([0,T],\bR^d)$. 
\end{proposition}
\begin{proof}
By changing measure, then using H\"older's inequality 
and Proposition \ref{proposition 6.20.1} 
we have 
$$
E|x_n(t)-x_n(s)|^4\leq\Big(\tilde E\gamma_{n}^{2}(T)\Big)^{1/2}
\Big(\tilde E|x_n(t)-x_n(s)|^{8}\Big)^{1/2}\leq N|t-s|^2
$$
for all $s,t\in[0,T]$, where the constant $N$ is independent of $s,t,n$ in light of Remark
\ref{remark 6.30.1}. This proves 
the proposition. 
\end{proof}
 
\mysection{Proof of Theorem \ref{Theorem 2.8}}

The reader can easily check that we can repeat the proof of
Theorem \ref{Theorem 2.4} given in Section  \ref{section 3},  
if we prove the following version of Lemma \ref{Lemma 3.1}.  
We use the same notations as in Section  \ref{section 3}.
\begin{lemma}
                                              \label{Lemma 5.1}
 Let $f(s,x)$ be a  Borel function defined on
$\Bbb{R}_{+}\times\Bbb{R}^{d}$ such that
$|f(t,x)|$ $\leq$ $ M_{k}(t)$ for any $k$ and $x\in D_{k}$. Then for
any $i=1,...,d_{1}$ the first two convergences in (3.6) hold as
$j\to\infty$ uniformly in $t\in[0,T\wedge\tilde{\tau}^{k})$ in   
probability for any $T<\infty$. If $|f(t,x)|^{2}\leq M_{k}(t)$ for
any $k$ and $x\in D_{k}$ and $t\leq k$, then for any $i=1,...,d_{1}$ the last two
convergences (3.6) also hold as $j\to\infty$ uniformly in
$t\in[0,T\wedge\tilde{\tau}^{k})$ in probability for any $T<\infty$.
\end{lemma}
 
\begin{proof} 
We will prove only the last relation in (3.6). The other
ones are considered in like manner. Take a continuous in $x$ Borel in
$t$
 function $g(t,x)$ defined on $\Bbb{R}_{+}\times\Bbb{R}^{d}$
satisfying the same hypotheses as $f$, and define
$$
I_{t}^{kj}(g)=\int_0^t g(s,\tilde
x_{n(j)}^{k}(\kappa_{n(j)}(s)))\,d\tilde w_{j}^{i}(s),\,\,\,
I_{t}^{k}(g)=\int_0^t g(s,\tilde x^{k}( s ))\,d\tilde w^{i}(s).
$$
Owing to Lemma \ref{Lemma 3.1} for any $\delta>0$ we have
\begin{equation}                                                                    \label{6.14.3}
\limsup_{j\to\infty}
P(\sup\{|I_{t}^{kj}(f)-I_{t}^{k}(f)|:\,
t<T\wedge\tilde{\tau}^{k}\}\geq
3\delta)
\end{equation}
$$
\leq\limsup_{j\to\infty}
P(\sup\{|I_{t}^{kj}(f-g)|:\,
t<T\wedge\tilde{\tau}^{k}\}\geq  \delta)
$$
$$
+ P(\sup\{|I_{t}^{k}(f-g)|:\,t<T\wedge\tilde{\tau}^{k}\}\geq 
\delta)=:J_{1}+J_{2}.       
$$ Now, by virtue of (3.4) and well-known martingale inequalities 
$$
J_{1}\leq\gamma^{-1}\limsup_{j\to\infty}E\int_{0}^{T\wedge\tilde{\tau}^{k}_{n(j)}
} |f-g|^{2}(s,\tilde
x_{n(j)}^{k}(\kappa_{n(j)}(s)))\,ds+{\gamma\over\delta^{2}}
$$
$$
\leq 4\gamma^{-1}\int_{0}^{\eta}M_{k}(s)\,ds+
\gamma^{-1}\limsup_{j\to\infty}\int_{\eta}^{T} E|(f-g)I_{D_{k}}|^{2}(s,\tilde
x_{n(j)}^{k}(\kappa_{n(j)}(s))) \,ds+{\gamma\over\delta^{2}},
$$ 
where $\gamma>0$ and $\eta>0$ are arbitrary numbers. By Corollary
4.3 we  conclude that for $p$ large enough
$$ 
J_{1}\leq
4\gamma^{-1}\int_{0}^{\eta}M_{k}(s)\,ds+{\gamma\over\delta^{2}}+
N\gamma^{-1}\Big[\int_{0}^{T}\int_{D_{k}}|f-g|^{2p}(s,x)\,dxds\Big]^{1/p}
$$ 
with $N$ independent of $g$. Since
$\tilde{x}^{k}_{n(j)}(t)\to\tilde{x}^{k}(t)$ (a.s.) from Corollary \ref{Corollary 4.3}
we also get an estimate for probability density of $x^{k}(t)$,
and this estimate shows that $J_{2}$ can be estimated by the  same
quantity as $J_{1}$. Thus we obtain an estimate for the first limit
in \eqref{6.14.3}, and this  estimate along with the freedom of choice of
$g,\eta,\gamma$
 shows that the limit in question is zero. This brings to an end the
proofs of Lemma \ref{Lemma 5.1} and Theorem \ref{Theorem 2.8}. 
\end{proof}
 
\mysection{Proof of Theorem \ref{theorem 6.21.5}}                                                         \label{Section 6}

To prove Theorem \ref{theorem 6.21.5} notice that by 
Proposition \ref{proposition 6.24.6} the sequence of Euler approximations 
$(x_n)_{n=1}^{\infty}$ is  tight  in $C([0,T],\bR^d)$. Thus to prove the theorem we need only prove  
the following lemma.

\begin{lemma}                                                                \label{lemma 6.24.10}
Let  Assumption 
\ref{assumption 7.5.1}  hold. 
Assume there exists a  process $x(t)$ such that for each $T>0$ 
$$
\lim_{n\to\infty}P(\sup_{t\in[0,T]}|x_n(t)-x(t)|\geq \varepsilon)=0
\quad\text{for each $\varepsilon>0$}.  
$$ 
Then for each $T>0$   
$$
\lim_{n\to\infty}E\int_0^T|b_n(t,x_n(\kappa_n(t)))-b(t,x(t))|^{2}\,dt=0.
$$
\end{lemma}

\begin{proof} To ease the notation we write $y_n(t)=x_n(\kappa_n(t))$.
For $m=1,2,...$
introduce
$$
b_{n}^{m}=b_{n}m/(m+ |b_{n}|),\quad
b ^{m}=b m/(m+ |b |).
$$
Note that $b_{n}\to b$ in measure, hence,
$b^{m}_{n}\to b^{m}$ in measure for each $m$,
and since all these functions are uniformly integrable (in $L_{2p,2 q ,T}$-sense),
$b^{m}_{n}\to b^{m}$ in $L_{2p,2 q ,T}$ for each $m$.
Observe also that for $\gamma$, which is strictly less than $( q -1)/ q $ but larger than the one in Assumption \ref{assumption 7.5.1}
(3), we have
$$
\lim_{n\to\infty}
d_n^{\gamma}(T)\|(b_{n}-b_{n}^{m})^{2}\|_{\infty, q ,T} \leq 4\lim_{n\to\infty}
d_n^{\gamma}(T)\| b_{n} ^{2}\|_{\infty, q ,T} =0.
$$
Hence and 
by Proposition \ref{proposition 6.21.1}
$$
 \limsup _{n\to \infty}E\int_0^T|b _{n}(t,y_n(t))
-b_{n}^{m}(t,y_n(t))|^{2}\,dt
$$
$$
\leq N\lim_{n\to \infty}(\|\,|b_{n}-b_{n}^{m}|^{2}\|_{p, q ,T}+d_n^{\gamma}(T)\|\,|b_{n}-b_{n}^{m}|^{2}\|_{\infty, q ,T})=N
\|\,|b -b ^{m}|^{2}\|_{p, q ,T},
$$
which can be made arbitrarily small if we choose
$m$ large enough.
It follows from here and Proposition \ref{proposition 6.24.2}
that to prove the lemma it suffices to prove that for each $m$
$$
\lim_{n\to\infty}E\int_0^T|b_{n}^{m} (t,y_n(t))-b^{m}(t,x(t))|^{2}\,dt=0.
$$
Furthermore, again by Proposition \ref{proposition 6.21.1}
$$
\lim_{n\to\infty}E\int_0^T|b_{n}^{m} (t,y_n(t))-b^{m}(t,y_n(t))|^{2}\,dt 
$$
$$
\leq N\lim_{n\to \infty}(\|\,|b_{n}^{m}-b ^{m}|^{2}\|_{p, q ,T}+d_n^{\gamma}(T)\|\,|b_{n}^{m}-b ^{m}|^{2}\|_{\infty, q ,T})
$$
\begin{equation}
                             \label{7.5.3}
\leq N\lim_{n\to \infty}4m^{2}T^{1/ q }d_n^{\gamma}(T)=0,
\end{equation}
which reduces the proof to showing that
for each $m$
$$
I:=\lim_{n\to\infty}E\int_0^T|b ^{m} (t,y_n(t))-b^{m}(t,x(t))|^{2}\,dt=0.
$$
Observe that  for any continuous   bounded $\bR^{d}$-valued
$\bar b(t,x)$ we obviously  have
$$
I\leq 9\lim_{n\to\infty}E \int_0^T|b ^{m} (t,y_n(t))-\bar b (t, y_n(t))|^{2}
+9E \int_0^T|b ^{m}(t,x(t))-\bar b(t, x(t))|^{2}
$$

Here the first term on the right is dominated by
\begin{equation}
                             \label{7.5.4}
N\|\,|b^{m}-\bar b |^{2}\|_{p, q ,T},
\end{equation}
which is proved similarly to \eqref{7.5.3}, and
the second term is dominated by the same expression in light of
Proposition \ref{proposition 6.24.2}.
After that it only remains to notice that
\eqref{7.5.4} can be made as small as we wish
for an appropriate $\bar b$ since
$C([0,T],C^{\infty}_{0}(\bR^{d}))$ is dense
in $L_{2p,2 q }(T)$ in light of the condition  $p, q <\infty$
(used for the first and the only time).
The lemma is proved. 
\end{proof}

 {\bf Acknowledgment.} 
  The authors are sincerely grateful to T. Shiga 
 and to the anonymous referee for drawing their attention 
 to paper \cite{KN_88}.   
 The main part of this  research was completed during
the stay of the first author at the University of Minnesota, and he
is grateful for the hospitality, excellent working  conditions and
for the inspiring scientific atmosphere of this institution.    
   The authors are also grateful
to Martin Dieckmann who found errors
  in the  original proof of Lemma 4.1  in 2013.

\end{document}